\documentclass{amsart}
\pdfoutput 1

\usepackage[all,hyperref]{bi-discrete}
\usepackage{tikz}
\usepackage{geometry}
\usepackage{pgfplots}
\pgfplotsset{width=10cm,compat=1.9}
\usepackage{enumitem}

\usepackage[backend=biber,maxbibnames=5,maxalphanames=5,style=alphabetic,bibencoding=utf8,giveninits,url=false,isbn=false]{biblatex}
\AtBeginBibliography{\small}

\def\tilde{\widetilde}
\begin{document}

\author{Anna Balci}
\address{Anna Balci, Department of Mathematical Analysis, Charles University in Prague, Sokolovská 49/83, 186 75 Praha 8, Czech Republic}
\email{akhripun@math.uni-bielefeld.de}

\author{Linus Behn}
\address{Linus Behn, University of Bielefeld, Universit\"atsstr. 25, 33615 Bielefeld, Germany}
\email{linus.behn@math.uni-bielefeld.de}

\author{Lars Diening}
\address{Lars Diening, University of  Bielefeld, Universit\"atsstr. 25, 33615 Bielefeld, Germany}
\email{lars.diening@uni-bielefeld.de}

\author{Johannes Storn}
\address{Johannes Storn, Faculty of Mathematics and Computer Science, Institute of Mathematics, Leipzig University, Augustusplatz 10, 04109 Leipzig, Germany}
\email{johannes.storn@uni-leipzig.de}

\keywords{%
	p-Laplace,
	p-harmonic maps,
	vectorial,
	infinity Laplacian,
	regularity,
	Hurwitz problem
}

\subjclass[2020]{
	35J60, 
	35B65, 
	11E39, 
	11E25, 
	35J92, 
	35J94, 
	35J47, 
}

\thanks{
	This work was partially funded by the Deutsche Forschungsgemeinschaft   (DFG, German Research Foundation) - SFB 1283/2 2021 - 317210226 and IRTG 2235 (Project 282638148). Additionally, the research of Anna Kh. Balci was supported by Charles University PRIMUS/24/SCI/020 and Research Centre program No. UNCE/24/SCI/005.
}

\renewcommand{\Re}{\operatorname{Re}}
\renewcommand{\Im}{\operatorname{Im}}

\begin{abstract}
	We construct explicit examples of $p$-harmonic maps $u:\mathbb{R}^n \to \mathbb{R}^N$. These are more irregular than the previously known examples and thus provide new upper bounds for the regularity of $p$-harmonic maps, including the case of $\infty$-harmonic maps. To optimize our approach, we utilize solutions of the Hurwitz problem from algebra.
\end{abstract}

\title{Examples of $p$-harmonic maps}

\maketitle



\section{Introduction}\label{sec:introduction}

In this paper we explicitly construct examples of vectorial $p$-harmonic functions $u:\RRn \to \RRN$. Our examples are more irregular than all previously known examples, providing new insights into the limitations of regularity for general $p$-harmonic maps.

For $p\in [1,\infty)$ and $n\geq 2$, a $p$-harmonic map is a weak solution $u:\RRn \to \RRN$  to the $p$-Laplace system
\begin{align}\label{eq:plap}
	-\Delta_p u \coloneqq -\divergence (\abs{\nabla u}^{p-2} \nabla u)=0.
\end{align}
Here we use the convention that $(\nabla u)_{ij}=\partial_i u_j$  and  $\divergence$ is the column-wise divergence operator;  $\abs{\nabla u}$ is the Frobenius norm of $\nabla u$. For $p=2$, \eqref{eq:plap} is the Laplace equation. In this case every solution is smooth. For $p\neq 2$, $u$ is in general only smooth in regions where $\nabla u$ vanishes nowhere. Around \emph{critical points} $x\in \RRn$ with $\nabla u (x)=0$ we only have that $u\in C^{1,\alpha}(\RRn,\RRN)$ for some $\alpha >0$. This was shown in \cite{Uralceva68,Evans82,Lewis83} for $N=1$ and in \cite{Uhlenbeck77,Tolksdorf84,AcerbiFusco89} for $N\geq 1$. 

The precise value of $\alpha$ depends on $p$ and $n$ but it is known only in the case  $n=2$ and $N=1$, see \cite{Dobrowolski85,Aronsson88,IwaniecManfredi89}. The optimal $\alpha$ such that $u\in C^{k,\beta}(\setR^2,\setR)$ with $k+\beta =1+\alpha$, $k\in \setN_{>0}$, and $\beta \in (0,1]$ is given by
\begin{align}\label{eq:alpha2d}
	\alpha = 
	\frac{1}{6}\Bigg(1+\frac{1}{p-1}+\sqrt{1+\frac{14}{p-1}+\frac{1}{(p-1)^2}}\,\Bigg).
\end{align}

The optimal value of $\alpha$ is unknown in all other cases~$n>2$. The goal of this paper is to construct families of $p$-harmonic functions in general dimensions, which provide upper bounds for $\alpha$.

Let us take a closer look at the regularity provided by \eqref{eq:alpha2d}, see Figure~\ref{fig:thetaalpha} (left). We observe that $\alpha \to \infty$ for $p \to 1$ and $\alpha \to \frac 13$ for $p\to \infty$. If instead of $\alpha$ we plot $\tau \coloneqq \alpha /p'=\alpha \frac{p-1}{p}$, then the graph becomes part of an ellipse, see Figure~\ref{fig:thetaalpha} (right).
We will observe similar ellipses in the regularity of our examples. Directly after Theorem~\ref{thm:cond_for_h} we explain in more detail, why~$\tau$ and $\theta=\frac 1p$ are natural quantities.
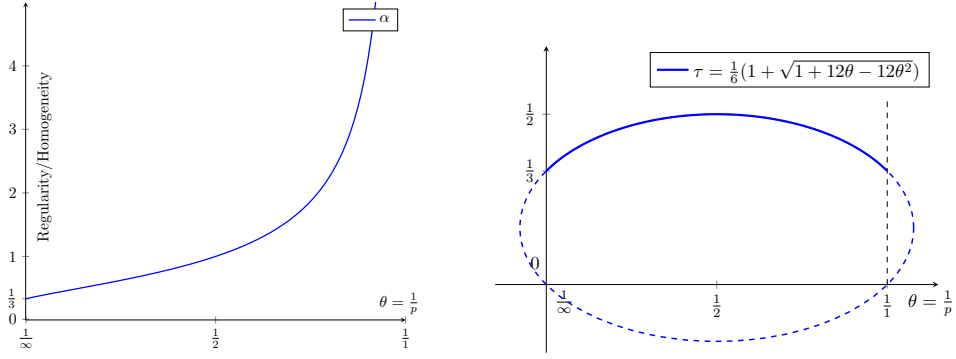
\begin{figure}[ht!]
	\label{fig:thetaalpha}
	\begin{tikzpicture}[scale=0.6]
		\begin{axis}[
			axis lines = middle,
			xmin=0,
			xmax=1,
			ymin=0,
			ymax=5,
			xtick={0.001,1/2,1},
			xticklabels={$\frac{1}{\infty}$,$\frac{1}{2}$,$\frac{1}{1}$},
			ytick={0.01,1/3,1,2,3,4},
			yticklabels={0,$\frac{1}{3}$,1,2,3,4},
			x label style = {at={(axis description cs: 1.05,0)}},
			xlabel = {\(\theta=\frac{1}{p}\)}
			]
			\node[left] at (axis cs:0,0) {0};
			\node[below right] at (axis cs:0,0) {$\frac{1}{\infty}$};
			\node[rotate=90, left] at (axis cs: 0.05,4) {Regularity/Homogeneity};
			\addplot [domain=0:0.95,samples=100,color=blue,thick]
			{((1+12*x-12*x^2)^(0.5)+1)/(6-6*x)};
			\addlegendentry{\(\alpha\)}
		\end{axis}
	\end{tikzpicture}
	\hspace{2em}
	\begin{tikzpicture}[scale=0.7]
		\begin{axis}[
			axis lines = middle,
			axis equal image,
			xmin=-0.15,
			xmax=1.15,
			ymin=-0.2,
			ymax=0.7,
			xtick={0,1/2,1},
			xticklabels={$\frac{1}{\infty}$,$\frac{1}{2}$,$\frac{1}{1}$},
			ytick={0,1/3,0.5},
			yticklabels={0,$\frac{1}{3}$,$\frac{1}{2}$},
			x label style = {at={(axis description cs: 1.05,0.1)}},
			xlabel = {\(\theta=\frac{1}{p}\)}
			]
			\node[below left] at (axis cs:0,0.1) {0};
			\node[below left] at (axis cs:0.1,0) {$\frac{1}{\infty}$};
			\draw[thick,blue,dashed] (axis cs: .5,1/6) ellipse ({sqrt(3)/3} and 1/2-1/6);
			\draw[black,dashed] (axis cs: 1,0) -- (axis cs: 1,0.55);
			\addplot [domain=0:1,samples=100,color=blue,very thick]
			{1/6*(1+sqrt(1+12*x-12*x*x))};
			\addlegendentry{\(\tau = \frac 16 (1+\sqrt{1+12\theta -12\theta ^2})\)}
		\end{axis}
	\end{tikzpicture}
	\caption{\small Optimal regularity of $p$-harmonic functions for $n=2$, $N=1$, as given in \eqref{eq:alpha2d}. Left: the regularity $\alpha$ of $\nabla u$. Right: $\tau=\alpha /p'$ forms an ellipse.}
\end{figure}

From this, we make the following observations about the case $n=2$, $N=1$:

\begin{enumerate}[label={(A\arabic{*})}]
	\setlength{\itemsep}{0.7em}
	\item \label{itm:conj_u2} The regularity of $\nabla u$ is decreasing in $p$, approaching infinity smooth functions for $p\to 1$ and $C^{1/3}$ regular functions for $p\to \infty$.
	\item \label{itm:conj_A2} The regularity of $A(\nabla u)$ is increasing in $p$, ranging from $C^{1/3}$ for $p\to 1$ to $C^{\infty}$ for $p=\infty$.
	\item \label{itm:conj_u}If $p\leq 2$, then $\nabla u\in C^1$.
	\item \label{itm:conj_A}If $p\geq 2$, then the stress $A(\nabla u)\coloneqq \abs{\nabla u}^{p-2}\nabla u$ satisfies $A(\nabla u)\in C^1$.
	\item \label{itm:conj_V}For all $p\in (1,\infty)$, the natural quantity $V(\nabla u)\coloneqq \abs{\nabla u}^{\frac{p-2}{2}}\nabla u$ satisfies $V(\nabla u)\in C^1$.
\end{enumerate}
These observations are shown again in Figure~\ref{fig:conjectures}.
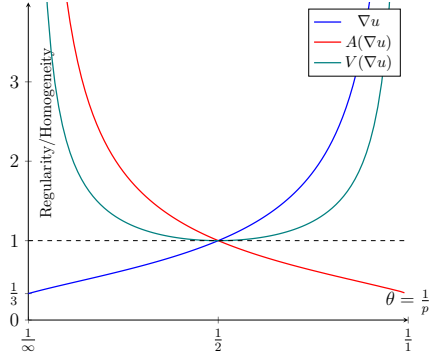
\begin{figure}[ht!]
	\begin{tikzpicture}[scale=0.6]
		\begin{axis}[axis lines = left, 
			ymin=0,
			ymax=4,
			label style={font=\Large},	
			ticklabel style={font=\Large},	
			xlabel = {\(\theta=\frac{1}{p}\)},xmin=0,xmax=1,ymin=0,
			xlabel style={at={(axis description cs:1,0)},anchor=south},	
			ytick={0,1/3,1,2,3},
			yticklabels={0,$\frac 13$,1,2,3},
			xtick={0,1/2,1},
			xticklabels={$\frac{1}{\infty}$,$\frac{1}{2}$,$\frac{1}{1}$}
			]
			\node[rotate=90, left] at (axis cs: 0.05,3.5) {Regularity/Homogeneity};
			\addplot [domain=0:0.9,samples=100,color=blue,thick]
			{((1+12*x-12*x^2)^(0.5)+1)/(6-6*x)};
			\addlegendentry{\(\nabla u\)}
			\addplot [domain=0:0.99,samples=100,color=red,thick]
			{((1+12*x-12*x^2)^(0.5)+1)/(6-6*x)*((1/x)-1)};
			\addlegendentry{\(A(\nabla u)\)}
			\addplot [domain=0:0.99,samples=100,color=teal,thick]
			{((1+12*x-12*x^2)^(0.5)+1)/(6-6*x)*(1/(2*x))};
			\addlegendentry{\(V(\nabla u)\)}
			\addplot [domain=0:1,dashed] coordinates {(0,1) (1,1)};
		\end{axis}
	\end{tikzpicture}
	\caption{\small The regularity of $\nabla u$, $A(\nabla u)=\abs{\nabla u}^{p-1}\nabla u$ and $V(\nabla u)=\abs{\nabla u}^{\frac{p-2}{2}}\nabla u$ for $n=2$ and $N=1$.}
	\label{fig:conjectures}
\end{figure}

The symmetry between the regularity of $\nabla u$ and $A(\nabla u)$ in Figure~\ref{fig:conjectures}, reflects the duality of $p$ and $p'$ harmonic functions in the plane, see \cite[Section 2]{Aronsson88}. This actually shows that observations \ref{itm:conj_u} and \ref{itm:conj_A} are equivalent. We also see that observation \ref{itm:conj_V} is the strongest and implies the other two. Note that \ref{itm:conj_A} implies that $u \in C^{p'}$ (for $p\geq 2$) which is known as the \emph{$p'$-conjecture}, see \cite{LindgrenLindqvist2017,ArujoTeixeiraUrbano2017}.

Naturally, the question arises whether these observations hold true for different values of $n$ and $N$. It was conjectured in \cite[Conjecture 2.28]{BalciDieningWeimar2020} that \ref{itm:conj_V} holds true for all $n,N$.

In this work we construct $p$-harmonic maps for all $n\geq 2$ and certain $N>1$. Their regularity depends on $p$ and $n$, but is independent of $N$. We now list for each of the observations \ref{itm:conj_u2}-\ref{itm:conj_V} above, whether our examples fulfill them:

\begin{enumerate}[label={(B\arabic{*})}]
	\setlength{\itemsep}{0.7em}
	\item \label{itm:u2} The regularity of $\nabla u$ is decreasing in $p$, ranging from $C^{1,\frac{2}{n-1}}$ for $p\to 1$ to $C^0$ for $p\to \infty$.
	\item \label{itm:A2} The regularity of $A(\nabla u)$ is increasing in $p$, ranging from $C^{0}$ for $p\to 1$ to $C^{n+1}$ for $p=\infty$.
	\item \label{itm:u} For $p\leq 2$, $\nabla u\in C^1$ holds and for every $\epsilon >0$ there exists an example such that $\nabla u \notin C^{1,\epsilon}$.
	\item \label{itm:A} For $p\geq 2$, $A(\nabla u) \in C^1$ holds for all our examples.
	\item \label{itm:V} For $p\in (1,2)$ and $n\geq 3$ we have $V(\nabla u)\notin C^1$. In particular, this disproves the conjecture from \cite{BalciDieningWeimar2020} mentioned above.
\end{enumerate}

Each of these properties is proven in Section \ref{sec:analysis}.

Let us present the concept behind our examples. They always have the structure
\begin{align*}
	u(x)=\abs{x}^{\gamma-2} h(x),
\end{align*}
where $\gamma\in \setR$ and $h:\RRn \to \RRN$ is a homogeneous polynomial of degree $2$. In Theorem~\ref{thm:cond_for_h} we find conditions on $h$ and $\gamma$ that are sufficient to make $u$ $p$-harmonic. For $h$ the conditions are
\begin{enumerate}
	\item [$(h1)$]\label{it:h1_intro} $\Delta h = 0$, 
	\item [$(h2)$]\label{it:h2_intro} $\abs{h(x)} = \abs{x}^{2}$ for all $x\in\RRn$.
\end{enumerate}
Whenever an $h$ fulfills those two conditions, we can determine a $\gamma$ such that $u$ is $p$-harmonic. The choice of $\gamma$ depends only on $p$ and $n$, but neither on $h$ nor $N$. It is given implicitly as the larger solution of the quadratic equation
\begin{align}
	\gamma^2 (p-1)+\gamma (n-p)-2n=0,
\end{align}
or explicitly as
\begin{align}
\gamma = \frac{p-n+\sqrt{(p-n)^2+8n(p-1)}}{2(p-1)}.
\end{align}
Since $u$ is $\gamma$-homogeneous, its regularity can by read off from the value of $\gamma$.

The main question to be answered now is for which combinations $(n,N)$ we can find a polynomial $h$ fulfilling \hyperref[it:h2_intro]{(h1)} and \hyperref[it:h2_intro]{(h2)}. Let us call such a pair $(n,N)$ \emph{admissible}. In Section~\ref{sec:explicit_examples}, we will show that for each $n\geq 2$, the pair $(n,n(n-1))$ is admissible. 

Note that if $(n,N)$ is admissible and $N_2>N$, then $(n,N_2)$ is admissible, too. Indeed, we can just add $N_2-N$ zero-components. This raises the question of finding for given $n$ the smallest $N$ which makes $(n,N)$ admissible. In Example~\ref{ex:small_n} we give examples of admissible pairs for small $n$. In particular, we show that $(2,2)$, $(3,5)$ and $(4,3)$ are admissible. In those three cases, the value of $N$ can not be decreased. 

In Section~\ref{sec:explicit_examples} we optimize the values of $N$, i.e., for given $n$ we seek for $N$ as small as possible, such that $(n,N)$ is admissible. This turns out to be directly related to the \emph{Hurwitz problem} about \emph{sums of squares formulas}. A sum of squares formula is an identity of the form
\begin{align}\label{eq:sos}
	(x_1^2+x_2^2+...+x_r^2)(y_1^2+y_2^2+...+y_s^2)=f_1^2+...+f_k^2.
\end{align}
In \eqref{eq:sos}, $r,s$ and $k$ are positive integers and the $f_i$ are bilinear forms in $x=(x_1,...,x_r)$ and $y=(y_1,...,y_s)$, i.e., $f_i(x,y)=\sum_{lm}a_{ilm}x_ly_m$. An elementary example of a sum of squares identity is
\begin{align*}
	(x_1^2+x_2^2)(y_1^2+y_2^2)=(x_1y_1-x_2y_2)^2+(x_1y_2+x_2y_1)^2.
\end{align*}
The Hurwitz problem asks for which values $r,s$ and $k$ such an identity can exist. Although many advancements have been made, this remains unsolved in full generality. For more information about the Hurwitz problem we refer to the book \cite{Shapiro2000Book} and the references therein. The details of how the Hurwitz problem is connected to our construction will be presented in Section~\ref{sec:explicit_examples}. We utilize this connection to get an asymptotic upper bound of $N \in O(\frac{n^2}{\log (n)})$. In the table below we present admissible pairs $(n,N)$ for selected values of $n$.

\renewcommand{\arraystretch}{1.8}
\vspace*{1em}
\begin{center}
	\begin{tabular}{|c||c|c|c|c|c|c|c|c|c|c|c|c|c|c|}
		\hline
		$n$ & $2$ & $3$ & $4$ & $5$ & $6$ & $7$ & $8$ & $10$ & $12$ & $14$ & $16$ & $n$ & even $n$ & large $n$
		\\
		\hline
		$N$ & $2$ & $5$ & $3$ & $8$ & $5$ & $12$ & $5$ & $9$ & $9$ & $9$ & $9$ & $n(n-1)$ & 1+$\frac{n^2}{4}$ & $O(\frac{n^2}{\log n})$
		\\
		\hline
	\end{tabular}
\end{center}
\vspace*{1em}

More admissible pairs can be found in Table \ref{table:admissible} below. In low, even dimensions $n$ ($n=2,\dots,20$) our example classes include vector fields $u:\RRn \to \RRn$. Of particular interest is the case $n=N=2$. As shown in Figure \eqref{fig:R2R2} (left), it is considerably less regular than the scalar example $n=2$ and $N=1$ which is known to be optimal. We have thus shown, that in $2$-d $p$-harmonic maps are in general less regular than $p$-harmonic functions.

In Figure~\ref{fig:R2R2} (right) we present the regularity of our examples for $n=2,3,4$. The case $n\geq 3$ disproves the $V \in C^1$-conjecture (observation \ref{itm:conj_V}), since the corresponding curves do not stay above the dashed line representing the $V\in C^1$-conjecture. Keep in mind that all our examples are vectorial, so that we do not make any statement about scalar $p$-harmonic functions.

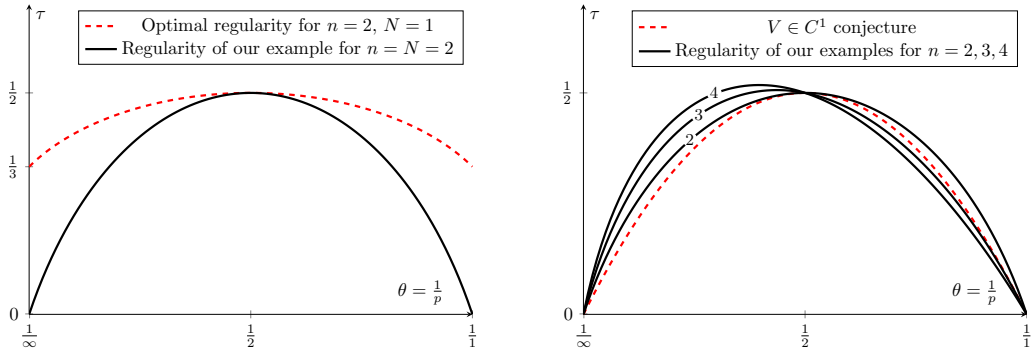
\begin{figure}[ht!]
	\begin{tikzpicture}[scale=0.7]
		\begin{axis}[
			axis lines = middle,
			axis equal image,
			xmin=-0,
			xmax=1,
			ymin=0,
			ymax=0.7,
			xtick={0.001,1/2,1},
			xticklabels={$\frac{1}{\infty}$,$\frac{1}{2}$,$\frac{1}{1}$},
			ytick={0.001,1/3,0.5},
			yticklabels={0,$\frac{1}{3}$,$\frac{1}{2}$},
			ylabel = {\(\tau\)},
			x label style = {at={(axis description cs: 0.95,0.01)}},
			xlabel = {\(\theta=\frac{1}{p}\)}
			]
			\node[below left] at (axis cs:0,0.1) {0};
			\node[below left] at (axis cs:0.1,0) {$\frac{1}{\infty}$};
			\addplot [domain=0:1,samples=100,color=red,dashed,very thick]
			{1/6*(1+sqrt(1+12*x-12*x*x))};
			\addlegendentry{Optimal regularity for $n=2$, $N=1$}
			\addplot [domain=0:1,samples=100,very thick]
			{1/2*(-1+sqrt(1+12*x-12*x*x))};
			\addlegendentry{Regularity of our example for $n=N=2$}
		\end{axis}
	\end{tikzpicture}
	\hspace{2em}
	\begin{tikzpicture}[scale=0.7]
		\begin{axis}[
			axis lines = middle,
			axis equal image,
			xmin=-0,
			xmax=1,
			ymin=0,
			ymax=0.7,
			xtick={0.001,1/2,1},
			xticklabels={$\frac{1}{\infty}$,$\frac{1}{2}$,$\frac{1}{1}$},
			ytick={0.001,0.5},
			yticklabels={0,$\frac{1}{2}$},
			ylabel = {\(\tau\)},
			x label style = {at={(axis description cs: 0.95,0.01)}},
			xlabel = {\(\theta=\frac{1}{p}\)}
			]
			\node[below left] at (axis cs:0,0.1) {0};
			\node[below left] at (axis cs:0.1,0) {$\frac{1}{\infty}$};
			\addplot [domain=0:1,samples=100,color=red,dashed,very thick]
			{2*(1-x)*x};
			\addlegendentry{$V\in C^1$ conjecture}
			\addplot [domain=0:1,samples=100,very thick]
			{1/2*(-1+sqrt(1+12*x-12*x*x))};
			\addplot [domain=0:1,samples=100,very thick]
			{1/2*(-x-1+sqrt(1+18*x-15*x*x))};
			\addplot [domain=0:1,samples=100,very thick]
			{1/2*(-2*x-1+sqrt(1+24*x-16*x*x))};
			\fill[white] (axis cs: 0.24,0.395) circle (4pt);
			\node at (axis cs: 0.24,0.395) {\small$2$};
			\fill[white] (axis cs: 0.26,0.45) circle (4pt);
			\node at (axis cs: 0.26,0.45) {\small$3$};
			\fill[white] (axis cs: 0.295,0.5) circle (4pt);
			\node at (axis cs: 0.295,0.5) {\small$4$};
			\addlegendentry{Regularity of our examples for $n=2,3,4$};
		\end{axis}
	\end{tikzpicture}
	\caption{\small Left: Regularity for $n=N=2$ expressed in $\tau=(\gamma -1)/p'$. It is more irregular than any scalar $p$-harmonic function on $\setR^2$. Right: Regularity of our examples for $n=2,3,4$. This disproves the $V\in C^1$-conjecture for $n\geq 3$.}
	\label{fig:R2R2}
\end{figure}

The rest of the paper is structured as follows. In Section \ref{sec:general_construction} we present our general construction and prove the $p$-harmonicity. In Section \ref{sec:explicit_examples} we provide explicit examples for all dimensions $n\geq 2$. Much of this section is concerned with the optimization of $N$ for given $n$. In Section \ref{sec:analysis} we prove the observations about the regularity of our examples from above. In Section \ref{sec:higher_order} we present a class of examples build around the use of higher order polynomials and in Section \ref{sec:infinity_laplacian} we discuss the case of infinity-harmonic maps.

\section{\texorpdfstring{Construction of $p$-harmonic maps}{Construction of p-harmonic maps}}\label{sec:cosntruction}

In this section we introduce new examples of $p$-harmonic maps. These maps will give us insight about the regularity of general $p$-harmonic maps. Certainly, the worst general solution can not be more regular than our maps.


\subsection{The general construction}\label{sec:general_construction}

We start with a bit of standard notation. Since we work in a vectorial setting it is particularly important to use consistent notation for matrices, maps, and gradients. Points $x\in \RRn$ are always denoted as row vectors $x=(x_1,...,x_n)$. Maps $u:\RRn \to \RRN$ are denoted as row vectors $u=(u_1,...,u_N)$, too. The gradient of $\nabla u$ of $u$ is the matrix $(\partial_i u_j)_{i,j}\in \setR ^{n\times N}$. In particular, the gradient of a scalar function is a column vector. As usual, let $L^p$ and $W^{k,p}$ denote the Lebesgue and Sobolev spaces, respectively. A function $u$ is called $\gamma$-homogeneous if $u(\lambda x)=\lambda^\gamma u(x)$ for all $\lambda > 0$.

For $a,b\in \RRn$ we write $a\cdot b = a b^T\in \setR$ for the scalar product of $a$ and $b$. For $a\in \RRn$ and $c\in \RRN$, we write $a\otimes c \coloneqq a^Tc\in \RRn \times \RRN$ for the tensor product of $a$ and $c$. For two matrices $M,N\in \setR^{n\times N}$, the Frobenius inner product is defined via $M:N \coloneqq \sum_{i,j} M_{ij}N_{ij}$. By $\abs{M}$ we always mean the Frobenius norm $\abs{M}\coloneqq \sqrt{M\!:\!M}$.

We now define weak solutions of the $p$-Laplace system for $p\in [1,\infty)$. The case $p=\infty$ is discussed in Section~\ref{sec:infinity_laplacian}.

\begin{definition}[$p$-harmonic maps]\label{def:weak_solution}
	Let $p\in [1,\infty)$ and $n,N \in \setN_{>0}$. A map $u\in W^{1,p}_{\loc}(\RRn, \RRN)$ is called \emph{$p$-harmonic} if
	\begin{align*}
		\int_\RRn \abs{\nabla u}^{p-2}\nabla u : \nabla \phi \,dx = 0,
	\end{align*}
	for every $\phi \in C_c^\infty (\RRn, \RRN)$. We call $A(\nabla u) \coloneqq \abs{\nabla u}^{p-2} \nabla u$ the stress. Finally, we define the natural quantity $V(\nabla u)\coloneqq \abs{\nabla u}^{\frac{p-2}{2}}\nabla u$.
\end{definition}
We use the ansatz
\begin{align}\label{eq:ansatz}
	u(x)=\abs{x}^{\gamma-k}h(x),
\end{align}
where $k\in \setN$, $h:\RRn \to \RRN$ is a suitable harmonic homogeneous polynomial of degree $k$. Moreover, $\gamma >0$ is chosen in dependence of $n$, $k$, and $p$.

The map $u$ is $\gamma$-homogeneous. This allows to easily read off the regularity of $u$. Homogeneity and regularity are strongly interlinked. In this paper we use the term regularity, even though strictly speaking we should refer to homogeneity. This makes a difference only in rare cases. If $\gamma -k$ is a nonnegative, even integer, then $u\in C^\infty$, although the homogeneity is $\gamma$. 

In the following theorem we establish sufficient conditions on $h$ which ensure that $u$ is $p$-harmonic. This enables us to construct new examples of $p$-harmonic maps for all $n\geq 2$. The existence of such $h$ will be investigated in detail in Section \ref{sec:explicit_examples}.

\begin{theorem}\label{thm:cond_for_h}
	Let $n\geq 2$ and $h:\RRn \to \RRN$ be a homogeneous polynomial of degree $k\in\setN _{>0}$ fulfilling the following two conditions
	\begin{enumerate}
		\item [$(h1)$]\label{it:h1} $\Delta h = 0$,
		\item [$(h2)$]\label{it:h2} $\abs{h(x)} = \abs{x}^{k}$ for all $x\in\RRn$.
	\end{enumerate}
	Let $p\in [1,\infty)$ and let $\gamma=\gamma(p,k,n)\in \setR$ be defined as follows. For $p=1$, let $\gamma =\frac{k(k+n-2)}{n-1}$. For $p\in(1,\infty)$, let $\gamma$ be the larger root of the quadratic polynomial
	\begin{align}\label{eq:gamma}
		\gamma^2 (p-1)+\gamma (n-p)-k(k+n-2),
	\end{align}
	i.e.,
	\begin{align}
    \label{eq:gamma2}
		\gamma = \frac{p-n+\sqrt{(p-n)^2+4(p-1)k(k+n-2)}}{2(p-1)}.
	\end{align}
	Then $u(x)\coloneqq\abs{x}^{\gamma-k}h(x)$ is a $p$-harmonic map. Moreover, $A(\nabla u) \in L^\infty_{\loc}(\RRn) \cap W^{1,1}_{\loc}(\RRn)$.
\end{theorem}

Before we proceed with the proof, we make a few short remarks about the value $\gamma$. First note that $\lim _{p\to 1} \gamma (p,k,n)= \gamma(1,k,n)$ and $\lim_{p\to \infty} \gamma (p,k,n)=1$. Furthermore, $\gamma$ is strictly monotone in $k$. An easy calculation shows that $\gamma(p,1,n)=1$. Thus, we always have $\gamma \geq 1$. If $k=1$, then $u=h$, where $h$ is a linear polynomial. Therefore, the first interesting case is $k=2$.

The regularity of~$u$ can be expressed more elegantly in other quantities. Indeed, if we define
\begin{align}\
  \label{eq:def-tau-theta}
  \tau &\coloneqq \frac{\gamma-1}{p'} = (1-\theta)(\gamma-1)
  \qquad \text{and} \quad   \theta \coloneqq \frac{1}{p},
\end{align}
then we can rewrite~\eqref{eq:gamma} in terms of~$\tau$ and $\theta$ as follows.
\begin{align}
  \label{eq:tau}
  \tau ^2 + \tau (1+\theta (n-2))+\theta(1-\theta)(n-1-k(k+n-2))=0.
\end{align}
In the important case $k=2$ we obtain
\begin{align}
  \label{eq:tau-k=2}
  \tau ^2 + \tau (1+\theta (n-2))-\theta(1-\theta)(n+1)=0.
\end{align}
The solution set to this equation is just an ellipse, see for example Figure~\ref{fig:thetaalpha} (right). From this ellipse we always need the upper part restricted to the interval $\theta \in [0,1]$. Let us explain, why the quantities $\tau$ and $\theta$ are natural. First of all, $\gamma-1$ is more natural than~$\gamma$. It is the homogeneity of~$\nabla y$ and $\nabla u$ appears naturally in the equation. Now, $\tau$ scales $\gamma-1$ by the factor~$\frac{1}{p'}$. This is very natural in the context of $n=2$ and $N=1$. In fact any $p$-harmonic functions in the plane induces another $p'$-harmonic function constructed by means of the stream function as done in~\cite{AronssonLindqvist88,Aronsson88}. This roughly speaking corresponds to switching $\nabla u$ with $A(\nabla u)$ and $p$ with $p'$. Now, the quantity $\tau=\frac{\gamma-1}{p'}$ is invariant under this transformation, since $\frac{(\gamma-1)(p-1)}{p} = \frac{\gamma-1}{p'}=\tau$. This symmetry is the reason, why in Figure~\ref{fig:thetaalpha} the graph is symmetric with respect to~$\theta = \frac 12$.

\begin{proof}[Proof of Theorem~\ref{thm:cond_for_h}]
  We first claim that \hyperref[it:h1]{(h1)} and \hyperref[it:h2]{(h2)} imply
	\begin{enumerate}
		\item [$(h3)$]\label{it:h3} $\abs{\nabla h(x)}=\sqrt{kn+k(2k-2)}\abs{x}^{k-1}$ and
		\item [$(h4)$]\label{it:h4} $x\nabla h(x)=kh(x)$.
	\end{enumerate}
	To see this, we note that by \hyperref[it:h2]{(h2)} and \hyperref[it:h1]{(h1)},
	\begin{align*}
		\Delta \abs{x}^{2k}=\Delta \abs{h}^2 = 2\skp{\Delta h}{h}+ 2 \abs{\nabla h}^2 = 2\abs{\nabla h}^2,
	\end{align*}
	and, by a basic computation,
	\begin{align*}
		\Delta \abs{x}^{2k}=(2kn+2k(2k-2))\abs{x}^{2k-2}.
	\end{align*}
	Thus, \hyperref[it:h3]{(h3)} follows. To show \hyperref[it:h4]{(h4)}, we fix $j\in \set{1,\dots,N}$. The function $h_j$ can be written as $h_j(x)=\sum_{\abs{\alpha}=k}\lambda_\alpha x^\alpha$, where the sum is taken over multi-indices of length $k$. Then
	\begin{align*}
		(x\nabla h(x))_j &= \sum_{i=1}^n x_i \partial_i (h_j(x))=\sum_{i=1}^n \sum_{\abs{\alpha}=k} \lambda_\alpha x_i \alpha_i x^{\alpha -e_i} = k h_j(x). 
	\end{align*}
	This shows \hyperref[it:h4]{(h4)} (which even holds for all homogeneous polynomials of degree $k$).

	Clearly, $u$ is smooth in $\RRn \setminus \set{0}$.
  Fixing $x\neq 0$, we now show that $\Delta_p u(x)=0$. We have
	\begin{align}
    \label{eq:grad-u}
		\nabla u(x)=(\gamma - k)\abs{x}^{\gamma - k -2}x\otimes h(x)+\abs{x}^{\gamma -k}\nabla h(x).
	\end{align}
	Since $\abs{x\otimes h(x)}^2=\abs{x}^2\abs{h(x)}^2$ and by \hyperref[it:h4]{(h2)}, \hyperref[it:h4]{(h3)} and \hyperref[it:h4]{(h4)} we get
	\begin{align}\label{eq:abs-nabla-u}
		\begin{aligned}
			\abs{\nabla u(x)}^2&= \abs{x}^{2\gamma-2k-2}\big((\gamma - k)^2 \abs{h(x)}^2+\abs{x}^{2}\abs{\nabla h(x)}^2 
			\\
			&\hspace{6.4em}+ 2(\gamma -k)x\nabla h(x)h(x)^T\big)
			\\
			&=\abs{x}^{2(\gamma -1)}\big( (\gamma - k)^2 + (kn+k(2k-2))+2(\gamma -k)k\big)
			\\
			&=\abs{x}^{2(\gamma -1)}(\gamma ^2 + k(k+n-2)).
		\end{aligned}
	\end{align}
	Setting $c\coloneqq \gamma ^2 + k(k+n-2)$ and $r\coloneqq \gamma p-p-\gamma -k$, we calculate
	\begin{align}\label{eq:pLap-pointwise}
		\begin{aligned}
			c^{\frac{2-p}{2}}\Delta_p u(x)&=\divergence\big(\abs{x}^{(\gamma -1)(p-2)}\nabla u(x)\big)
			\\
			&=\divergence\big((\gamma -k)\abs{x}^{(\gamma -1)(p-2)+\gamma-k-2}x\otimes h(x)
			\\
			&\hspace{3.2em}+\abs{x}^{(\gamma-1)(p-2)+\gamma-k}\nabla h(x)\big)
			\\
			&=(\gamma - k)(\nabla \abs{x}^r)^Tx\otimes h(x) + (\gamma -k)\abs{x}^r\divergence(x\otimes h(x))
			\\
			&\quad + (\nabla \abs{x}^{r+2})^T\nabla h(x)
			\\
			&\eqqcolon \mathrm{I}+\mathrm{II}+\mathrm{III}.
		\end{aligned}
	\end{align}
	Here we already used $\Delta h = 0$. We have
	\begin{align*}
		\mathrm{I}=(\gamma - k)r\abs{x}^{r-2}x\,x^Th(x)=(\gamma -k)r\abs{x}^{r}h(x),
	\end{align*}
	and
	\begin{align*}
		\mathrm{II}=(\gamma -k) \abs{x}^{r}(nh(x)+x\nabla h(x))=(\gamma -k)(n+k) \abs{x}^{r}h(x),
	\end{align*}
	where in the last equality we used \hyperref[it:h4]{(h4)}. Finally, using \hyperref[it:h4]{(h4)} again, we have
	\begin{align*}
		\mathrm{III}=(r+2)\abs{x}^{r}x\nabla h(x)=k(r+2)\abs{x}^{r}h(x).
	\end{align*}
	Inserting our findings into \eqref{eq:pLap-pointwise} we get
	\begin{align}
    \label{eq:locally-p-harmonic}
    \begin{aligned}
      \Delta_p u(x)&=c^{\frac{p-2}{2}} \left((\gamma-k)(r+n+k)+k(r+2)\right)\abs{x}^{r}h(x)
      \\
      &=c^{\frac{p-2}{2}}\left( \gamma ^2(p-1)+\gamma (n-p)-k(n+k-2)\right)\abs{x}^rh(x)
      \\
      &=0,
    \end{aligned}
	\end{align}
  where we have used in the last step the choice of $\gamma$ in \eqref{eq:gamma} . We have thus shown the $u$ is pointwise $p$-harmonic for all $x\in \RRn \setminus \set{0}$.

  Since $u$ is $\gamma$-homogeneous with~$\gamma\geq 1$, we have $u \in W^{1,\infty}_{\loc}(\RRn)$ and its weak derivative is given by~\eqref{eq:grad-u}. Moreover, $A(\nabla u)$ is $(\gamma-1)(p-1)$-homogeneous and $(\gamma-1)(p-1)\geq 0$. Thus, $A(\nabla u) \in L^\infty_{\loc}(\RRn) \cap W^{1,1}_{\loc}(\RRn)$. By~\eqref{eq:locally-p-harmonic} we have $\divergence A(\nabla u)$ on $\RRn \setminus \set{0}$. This and $A(\nabla u) \in W^{1,1}_{\loc}(\RRn)$ proves $\Delta_p u = \divergence A(\nabla u) =0$ in the weak sense. This finishes the proof.
\end{proof}


\subsection{Explicit examples}\label{sec:explicit_examples}

In this section we present explicit examples of homogeneous polynomials $h$ fulfilling the conditions of Theorem~\ref{thm:cond_for_h}. These $h$, combined with Theorem~\ref{thm:cond_for_h} directly give examples of $p$-harmonic maps. In this section we focus on the case $k=2$, because the $h$ with the lowest degree will yield the least regular $u$. The case $k=1$ only produces linear solutions $u$. The case $k>2$ is discussed in Section~\ref{sec:higher_order}.

For $n=N=2$, such an $h$ is given by
\begin{align}
  \label{eq:n=N=2}
	h(x)= \Bigg(
	\begin{matrix}
		x_1^2-x_2^2
		\\
		2x_1x_2
	\end{matrix}
	\Bigg).
\end{align}

We are now interested in the case $n\geq 2$. In particular, we want to find small values of $N$ that allow for a construction of $h$.

\begin{definition}\label{def:admissible_pair}
	A pair $(n,N)$ of positive integers with $n\geq 2$ is called \emph{admissible}, if there exists a homogeneous polynomial $h:\RRn \to \RRN$ of degree $2$ such that the conditions from Theorem~\ref{thm:cond_for_h} are fulfilled, i.e.,
	\begin{itemize}
		\item $\Delta h =0$ and
		\item $\abs{h(x)}=\abs{x}^2$ for all $x\in \RRn$.
	\end{itemize}
	By $N_{\min}(n)$, we denote the smallest $N$ such that $(n,N)$ is admissible.
\end{definition}

Not all pairs $(n,N)$ are admissible. For example, it is easy to see that a pair of the form $(n,1)$ is never admissible. Admissibility is a monotone property in the sense that if $(n,N)$ is admissible, then $(n,N')$ is admissible for every $N'>N$, too. This can be achieved by adding $(N'-N)$ zero-components to $h$. Such a monotonicity does not hold in the first component $n$.

For each $n$ we have $N_{\min} (n)\leq n(n-1)$, as the following construction shows. In particular, for each $n$ there exists an $N$ that make $(n,N)$ admissible. Let $h:\RRn \to \setR^{n(n-1)}$ be given by
\begin{align}
  \label{eq:h-simple-all}
  h &=
  \begin{pmatrix}
    x_i^2 - x_j^2 \quad \text{for $i<j$}
    \\
    2x_i x_j \quad \text{for $i<j$}
  \end{pmatrix}.
\end{align}
Then $h$ is harmonic and $\abs{h}^2 = \abs{x}^4$ as desired in Theorem~\ref{thm:cond_for_h}. The case $n=2$ recovers~\eqref{eq:n=N=2}. Moreover, we see that $N_{\min}(3) \leq 6$ and $N_{\min}(4) \leq 12$.

We now show how for even $n$ the construction can be improved with a smaller~$N$ such that $(n,N)$ is admissible. Indeed, for $n=2m$ let $h:\RRn \to \setR^{1+\frac{n^2}{4}}$ given by
\begin{align}
  \label{eq:h-simple-even}
  h(x)=
  \begin{pmatrix}
    (x_1^2+ \dots x_m^2) - (x_{m+1}^2 + \dots +x_{2m}^2)
    \\
    2x_i x_j \quad \text{for $1\leq i\leq m$, $m+1 \leq j \leq 2m$.}
  \end{pmatrix}
\end{align}
Then $h$ fulfills the conditions of Theorem~\ref{thm:cond_for_h}, so $(n,1+\frac{n^2}{4})$ is admissible for even~$n$. We obtain $N_{\min}(4) \leq 5$, which is an improvement to~\eqref{eq:h-simple-all}.

However, for $n=4$ it is possible to find even smaller $N$. In particular, $(4,3)$ is admissible using the function $h\,:\, \RR^4 \to \RR^3$ given by
\begin{align}
  \label{eq:n=4N=3}
  h(x)=
  \begin{pmatrix}
    x_1^2 + x_2^2 -x_3^2 - x_4^2
    \\
    2x_1x_3+2x_2x_4
    \\
    2x_1x_4-2x_2x_3
  \end{pmatrix}
  .
\end{align}
This shows that $N_{\min}(4) \leq 3$.

In order to construct suitable~$h\,:\, \RRn \to \RRN$ for small values of~$N$ we need some tools from algebra. In particular, we need \emph{compositions of quadratic forms} also known as the \emph{Hurwitz problem}. We refer to the book~\cite{Shapiro2000Book} for an extensive survey on this topic. The following notation and statements are from this book.

The Hurwitz problem is the following: For $r,s \in \setN_{>0}$ find $t \in \setN_{>0}$ and a bilinear map $F\,:\, \RR^r \times \RR^s \to \RR^t$ such that
\begin{align}
  \label{eq:hurwitz}
  \abs{F(x,y)} = \abs{x} \abs{y}.
\end{align}
If such $F$ exsists, then we say that triple $[r,s,t]$ is Hurwitz-admissible. By $t_{\min}(r,s)$ we denote the smallest~$t$ such that $[r,s,t]$ is Hurwitz-admissible. By $F_{r,s}$ we denote any bilinear map associated to $[r,s,t_{\min}(r,s)]$ satisfying~\eqref{eq:hurwitz}.

We will use~\eqref{eq:hurwitz} in the form
\begin{align}
  \label{eq:hurwitz2}
  \big(x_1^2 + \dots +x_r^2\big)^2 \big(y_1^2 + \dots + y_r^2\big)^2 = \big(F_1(x,y)\big)^2 + \dots + \big(F_t(x,y)\big)^2.
\end{align}

Using the Hurwitz-problem, we can now construct more examples of~$h$ as given in Theorem~\ref{thm:cond_for_h}.
\begin{theorem}
  \label{thm:N-hurwitz-even}
  Let $n =2m\in \setN_{>0}$ be even. Then $N_{\min} \leq 1 + t_{\min}(m,m)$ and
  \begin{align*}
    h(x) &=
    \begin{pmatrix}
      (x_1^2+ \dots x_m^2) - (x_{m+1}^2 + \dots +x_{2m}^2)
      \\
      2F_{m,m}\big((x_1,\dots,x_m),(x_{m+1},\dots, x_{2m})\big)
    \end{pmatrix}
  \end{align*}
  satisfies the assumption of Theorem~\ref{thm:cond_for_h}.
\end{theorem}
\begin{proof}
  We decompose $x$ into $x= (y,z)$ with $y=(x_1,\dots, x_m)$ and $z=(x_{m+1},\dots, x_{2m})$. Then
  \begin{align*}
    \abs{x}^4 &= \abs{(y,z)}^4 = \big(\abs{y}^2 -\abs{y}^2\big)^2 + 4 \abs{y}^2\abs{z}^2
    \\
    &= \big((y_1^2+ \dots y_m^2) - (z_1^2 + \dots +z_m^2)\big)^2 + \big(2F_{m,m}(y,z)\big)^2.
  \end{align*}
  Now, the claim follows, since $(y_1^2+ \dots y_m^2) - (z_1^2 + \dots +z_m^2)$ and the components of $F_{m,m}(y,z)$ are harmonic.
\end{proof}
The case that $n=2m+1$ is odd is more challenging. However, it can be reduced to the case $n-1=2m$ as the following Lemma shows.
\begin{theorem}
  \label{thm:N-hurwitz-odd}
  Let $n =2m+1\in \setN_{>0}$ be odd. Then $N_{\min}(n) \leq 1+t_{\min}\big(\frac{n-1}{2},\frac{n-1}{2}\big)+n$ and
  \begin{align*}
    h(x) &= \frac{1}{2m}
    \begin{pmatrix}
      \sqrt{4m^2-1} \big((x_1^2+ \dots x_m^2) - (x_{m+1}^2 + \dots +x_{2m}^2)\big)
      \\
      (x_1^2 + \dots + x_{2m}^2) - 2m x_n^2
      \\
      2\sqrt{4m^2-1}\,F_{m,m}\big((x_1,\dots,x_m),(x_{m+1},\dots, x_{2m})\big)
      \\
      2\sqrt{2m^2+m}\, x_ix_n \quad \text{for $i=1,\dots, 2m$}.
    \end{pmatrix}
  \end{align*}
  satisfies the assumption of Theorem~\ref{thm:cond_for_h}.
\end{theorem}
\begin{proof}
  We decompose $x$ into $x= (y,z,x_n)$ with $y=(x_1,\dots, x_m)$ and $z=(x_{m+1},\dots, x_{2m})$. Then $\abs{y}^2 - \abs{z}^2$ and $\abs{(y,z)}^2 - 2m x_n^2$ are harmonic. Moreover,
  \begin{align*}
    A^2 &\coloneqq \big(\abs{y}^2-\abs{z}\big)^2
    = \abs{y}^4 - 2 \abs{y}^2\abs{z}^2 +\abs{z}^4
    \\
    B^2 &\coloneqq \big( \abs{(y,z)}^2 - 2m x_n^2\big)^2
    = \abs{y}^4 + \abs{z}^4 + 4m^2 \abs{x_n}^4 +2 \abs{y}^2\abs{z}^2 - 4m \abs{(y,z)}^2x_n^2.
  \end{align*}
  Thus,
  \begin{align*}
    (4m^2-1) A^2 + B^2
    &= 4m^2 \big(\abs{y}^4 + \abs{z}^4 + x_n^4) + (-2(4m^2-1)+2)\abs{y}^2\abs{z}^2   - 4m \abs{(y,z)}^2x_n^2
    \\
    &= 4m^2 \big(\abs{y}^4 + \abs{z}^4 + x_n^4\big) -(8m^2-4)\abs{y}^2\abs{z}^2 - 4m \abs{(y,z)}^2x_n^2.
  \end{align*}
  Hence,
  \begin{align*}
    4m^2 \abs{x}^4 &= 4m^2\abs{(y,z,x_n)}^4
    \\
    &= (4m^2-1)A^2 + B^2 +(16m^2-4) \abs{y}^2 \abs{z}^2 +(8m^2+4m)  \abs{(y,z)}^2x_n^2
    \\
    &= (4m^2-1)A^2 + B^2 +(16m^2-4)\big(F_{m,m}(y,z)\big)^2 + (8m^2+4m) \Big( \sum_{i=1}^m y_i^2x_n^2 + \sum_{i=1}^m z_i^2x_n^2\Big).
  \end{align*}
  This and the fact that $F_{m,m}(y,z)$ and the $y_i^2x_n^2$ and $z_i^2x_n^2$ are harmonic prove the claim.
\end{proof}

Combining Theorem \ref{thm:N-hurwitz-even} and Theorem \ref{thm:N-hurwitz-odd} with the known upper bounds for $t_{\min}(n,n)$ (see \cite{Shapiro2000Book}), we collect admissible pairs $(n,N)$ in Table \ref{table:admissible}.

\renewcommand{\arraystretch}{1.8}
	\begin{table}[ht!]
		\label{table:admissible}
		\centering
		\begin{tabular}{|c||c|c|c|c|c|c|c|c|c|c|c|c|c|c|c|c|}
		\hline
		$n$ & $2$ & $4$ & $6$ & $8$ & $10$ & $12$ & $14$ & $16$ & $18$ & $20$ & $22$ & $24$ & $26$ & $28$ & $30$ & $32$
		\\
		\hline
		$N$ & $2$ & $3$ & $5$ & $5$ & $9$ & $9$ & $9$ & $9$ & $17$ & $17$ & $27$ & $27$ & $29$ & $33$ & $33$ & $33$
		\\
		\hline
	\end{tabular}

	\vspace*{1em}

	\begin{tabular}{|c||c|c|c|c|c|c|c|c|c|c|c|c|c|c|c|}
		\hline
		$n$ & $3$ & $5$ & $7$ & $9$ & $11$ & $13$ & $15$ & $17$ & $19$ & $21$ & $23$ & $25$ & $27$ & $29$ & $31$
		\\
		\hline
		$N$ & $5$ & $8$ & $12$ & $14$ & $20$ & $22$ & $24$ & $26$ & $36$ & $38$ & $50$ & $52$ & $56$ & $62$ & $64$
		\\
		\hline
	\end{tabular}

	\vspace{1em}
	\caption{Admissible pairs $(n,N)$ for $n=2,\dots,32$.}
	\end{table}
\vspace*{1em}

\begin{remark}
  \label{rem:asymptotics}
  Theorem~\ref{thm:N-hurwitz-even} and~\ref{thm:N-hurwitz-odd} give a good upper bound for the choice of~$N$ such that $(n,N)$ is admissible. It is interesting to know the asymptotic growth of $N_{\min}(n)$. This basically depends on the behavior of $t_{\min}(m,m)$.  Currently, to the best of our knowledge, the values $t_{\min}(m,m)$ are only known for small values of $n$. In 1898, Hurwitz showed that the triple $[n,n,n]$ is Hurwitz-admissible if and only if $n\in\set{1,2,4,8}$, see \cite{Hurwitz1963}. The best known asymptotic upper bound is $t_{\min}(m,m)\in O(m^2/\log (m))$. This asymptotic follows from the following observation: $[2k,2^k,2^k]$ is Hurwitz-admissible for~$k\geq 1$, see~\cite{Hurwitz1922} or \cite[page 4]{Shapiro2000Book}. This implies that $[2ak,2^k,a2^k]$ is Hurwitz-admissible for all~$a,k \geq 1$ by summation over the $a$ many bilinear maps $F_{2m,2^m}$. Now, choosing $a,k$ such that $2ak \approx 2^k \approx m$, gives $a \in O(\frac{m}{\log m})$ and $k \in O(m)$. This shows that $t_{\min}(m,m) \in O(\frac{m^2}{\log m})$. Using this, we can conclude  from Theorem~\ref{thm:N-hurwitz-even} and~\ref{thm:N-hurwitz-odd} that $N_{\min}(n) \in O(\frac{n^2}{\log n})$.
\end{remark}
We end this section by explicitly writing down some of the $p$-harmonic maps we obtained.
\begin{example}\label{ex:small_n}
	We give examples of $p$-harmonic maps for selected values of $n$ and all $p\in [1,\infty)$ based on Theorem \ref{thm:N-hurwitz-even} and \ref{thm:N-hurwitz-odd}. In the following, the exponent $\gamma$ is always given as in Theorem~\ref{thm:cond_for_h}, i.e., $\gamma = \frac{2n}{n-1}$ for $p=1$ and 
	\begin{align*}
		\gamma = \frac{p-n+\sqrt{(p-n)^2+8n(p-1)}}{2(p-1)}
	\end{align*}
	for $p>1$.
	\begin{enumerate}[label={(\roman{*})}]
		\item \textbf{$n=2$, $N=2$:} Choose 
		\begin{align*}
			u(x)= \abs{x}^{\gamma -2}
			\Bigg(\begin{matrix}
				x_1^2-x_2^2
				\\
				2x_1x_2
			\end{matrix}\Bigg)
			.
		\end{align*}
		\item \textbf{$n=3$, $N=5$:} Choose
		\begin{align*}
			u(x)=\tfrac{1}{2}\abs{x}^{\gamma -2}
			\begin{pmatrix}
				\sqrt{3}(x_1^2-x_2^2)
				\\
				x_1^2+x_2^2-2x_3^2
				\\
				2\sqrt{3}\,x_1x_2
				\\
				2\sqrt{3}\,x_1x_3
				\\
				2\sqrt{3}\,x_2x_3
			\end{pmatrix}
			.
		\end{align*}
		\item \textbf{$n=4$, $N=3$:} Choose 
		\begin{align*}
			u(x)=\abs{x}^{\gamma -2}
			\begin{pmatrix}
				x_1^2 + x_2^2 -x_3^2 - x_4^2
				\\
				2x_1x_3+2x_2x_4
				\\
				2x_1x_4-2x_2x_3
			\end{pmatrix}
			.
		\end{align*}
		\item \textbf{$n=5$, $N=8$:} Choose 
		\begin{align*}
			u(x)=\tfrac{1}{4}\abs{x}^{\gamma -2}
			\begin{pmatrix}
				\sqrt{15}(x_1^2+x_2^2-x_3^2-x_4^2)
				\\
				x_1^2+x_2^2+x_3^2+x_4^2-4x_5^2
				\\
				2\sqrt{15}(x_1x_3+x_2x_4)
				\\
				2\sqrt{15}(x_1x_4-x_2x_3)
				\\
				2\sqrt{10}x_ix_5 \quad \text{for $i=1,\dots,4$}
			\end{pmatrix}
			.
		\end{align*}
	\end{enumerate} 
\end{example}
\begin{remark}
  \label{rem:smallest}
  We verified by hand that our bounds for $N_{\min}(n)$ for $n=2,3,4$ cannot be improved, i.e., $N_{\min}(2)= 2$, $N_{\min}(3)=5$, $N_{\min}(4)=3$. However, we omit the tedious proof.
\end{remark}

\subsection{Analysis of our examples}\label{sec:analysis}

Let use repeat the observations about our examples as stated in the introduction. We are now able to justify these claims. Note that in this setting~$k=2$. We use the formulas~\eqref{eq:gamma}, \eqref{eq:gamma2} and~\eqref{eq:tau-k=2} for $\gamma$ and~$\tau$ and $n\geq 2$.
\begin{enumerate}[label={(B\arabic{*})}]
	\setlength{\itemsep}{0.7em}
	\item \label{itm:x-u2} The regularity of $\nabla u$ is decreasing in $p$, ranging from $C^{1,\frac{2}{n-1}}$ for $p\to 1$ to $C^0$ for $p\to \infty$.
	\item \label{itm:x-A2} The regularity of $A(\nabla u)$ is increasing in $p$, ranging from $C^{0}$ for $p\to 1$ to $C^{n+1}$ for $p=\infty$.
	\item \label{itm:x-u} For $p\leq 2$, $\nabla u\in C^1$ holds, but for every $\epsilon >0$ there exists an example such that $\nabla u \notin C^{1,\epsilon}$.
	\item \label{itm:x-A} For $p\geq 2$, $A(\nabla u) \in C^1$ holds for all our examples.
	\item \label{itm:x-V} For $p\in (1,2)$ and $n\geq 3$ we have $V(\nabla u)\notin C^1$. In particular, this disproves the conjecture from \cite{BalciDieningWeimar2020} mentioned above.
\end{enumerate}

\begin{proof}[Proof of \ref{itm:x-u2}-\ref{itm:x-V}]
	We begin with~\ref{itm:x-u2}. Standard calculations show that~$\gamma(p)$ is decreasing in~$p$. A simple way to see this is to differentiate the implicit equation \eqref{eq:gamma} in direction of $p$ to obtain
\begin{align*}
	\frac{d\gamma}{dp} (2\gamma(p-1)+n-p) + (\gamma^2 - \gamma)=0.
\end{align*}
Since $\gamma >1$ and $n\geq 2$ this proves $\frac{d\gamma}{dp}<0$. It is easy to see that $\gamma(p=1)=\frac{2n}{n-1}$ and $\gamma(p=\infty)=1$. This shows \ref{itm:x-u2}.

We continue with~\ref{itm:x-A2}. Let $a$ be the homogeneity of~$A(\nabla u)$. Then $a = (\gamma-1)(p-1) = \frac{\tau}{\theta} \geq 0$. From \eqref{eq:tau-k=2} we get
\begin{align}
	a^2\theta + a (1+\theta (n-2))-(1-\theta)(n+1)=0.
\end{align}
Differentiating this equation in direction of $\theta$ gives us
\begin{align*}
	\frac{d a}{d \theta} (2a\theta +1+\theta (n-2))+(a^2 +a(n-2)+(n+1))=0,
\end{align*}
which shows $\frac{da}{d\theta}<0$. Again, it is easy to see that $a(p=1)=0$ and $a(p=\infty)=n+1$.

To show \ref{itm:x-u}, we first note that \ref{itm:x-u2} and $\gamma(p=2)=2$ imply that $\gamma > 2$ for $p< 2$. Thus $\nabla u\in C^1$ holds in this range. If we consider increasing dimensions~$n$, then $\nabla u$ becomes less regular for $1\leq p \leq 2$ and more regular for $p \geq 2$. This can be observed in Figure~\ref{fig:n_to_infty}. For $n \to \infty$, the regularity of~$\nabla u$ expressed in $\tau = \frac{\gamma-1}{p'}$ converges to $\tau = 1-\theta$. Indeed, if we divide~\eqref{eq:tau-k=2} by~$\theta\,n$ and let $n \to \infty$, we obtain the limit equation $\tau  = 1-\theta$. This verifies \ref{itm:x-u}.

Observation \ref{itm:x-A} follows from \ref{itm:x-A2} together with $a(p=2)=(\gamma(p=2)-1)(2-1)=1$.

Finally, \ref{itm:x-V}, which was already observed in Figure \ref{fig:R2R2}, can be seen as follows. First we differentiate \eqref{eq:gamma} in direction of $n$ to get
\begin{align*}
	\frac{d\gamma}{dn}\cdot \left(2\gamma (p-1)+n-p\right)= 2-\gamma.
\end{align*}
For $p\in (1,2)$ (which implies $\gamma >2$), this shows $\frac{d\gamma}{dn}<0$. Thus, for fixed $p\in (1,2)$, the homogeneity $\nu\coloneqq (\gamma -1)\frac{p}{2}$ of $V(\nabla u)$ is monotonically decreasing in $n$, too. Let us thus consider the case $n=3$. Note that $\nu =\tau \frac{p' p}{2}=\frac{\tau}{2\theta(1-\theta)}$. Using \eqref{eq:tau-k=2} and $n=3$, we obtain
\begin{align*}
	\nu < 1 &\iff \frac{-1-\theta +\sqrt{1+18\theta-15\theta^2}}{2}\frac{1}{2\theta(1-\theta)}<1
	\\
	&\iff -1-\theta+\sqrt{1+18\theta -15\theta ^2}\leq 4 \theta - 4 \theta ^2 
	\\
	&\iff 0 \leq 2\theta^3-5\theta^2+4\theta -1,
\end{align*}
which holds true on the interval $\theta \in (\frac 12,1)$. Thus $V(\nabla u)\notin C^1$ whenever $n\geq 3$ and $p\in (1,2)$. This finishes the proof.

\end{proof}
\begin{figure}[ht!]
	\begin{tikzpicture}[scale=0.7]
		\begin{axis}[
			axis lines = middle,
			axis equal image,
			xmin=-0,
			xmax=1,
			ymin=0,
			ymax=1.1,
			xtick={0.001,1/2,1},
			xticklabels={$\frac{1}{\infty}$,$\frac{1}{2}$,$\frac{1}{1}$},
			ytick={0.001,0.5,1},
			yticklabels={0,$\frac{1}{2}$,1},
			ylabel = {\(\tau\)},
			x label style = {at={(axis description cs: 0.88,0)}, anchor = north},
			xlabel = {\(\theta=\frac{1}{p}\)},
			legend entries={Regularity for selected $n$,,,,,,$\tau=1-\theta$} 
			]
			\node[below left] at (axis cs:0,0.1) {0};
			\node[below left] at (axis cs:0.1,0) {$\frac{1}{\infty}$};
			\foreach \n in {2,8,32,128,512,2048}{
				\addplot [domain=0:1,samples=1000,thick]
				{(-1-x*(\n-2))*0.5+sqrt(0.25*(1+x*(\n-2))*(1+x*(\n-2))+x*(1-x)*(\n+1))};
			}
			\addplot [domain=0:1,samples=2,dashed,color=red,thick]
			{1-x};
		\end{axis}
	\end{tikzpicture}
	\caption{\small Regularity of our examples dimensions $2,8,32,128,512,2048$. The regularity is expressed in $\tau =(\gamma -1)/p'$ via $\tau ^2 + \tau (1+\theta (n-2))-\theta(1-\theta)(n+1)=0$. For $n\to \infty$ the functions approach $\tau = (1-\theta)$ (dashed).}
	\label{fig:n_to_infty}
\end{figure}
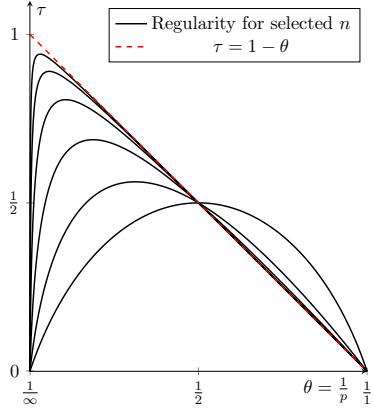


\subsection{Examples using higher order polynomials}\label{sec:higher_order}

While in the previous section we considered the case~$k=2$, we will now consider the case~$k \geq 2$. In particular, we will construct $k$-homogeneous polynomials~$h$ satisfying the assumption of Theorem~\ref{thm:cond_for_h}. Larger values of $k$ give more regular $p$-harmonic maps. We will restrict ourselves to the case of even dimension~$n$.

For our construction we need complex numbers and the fact that the real and imaginary part of harmonic functions are harmonic. This fact for multivariate complex functions is proved in the following lemma.
\begin{lemma}
  \label{lem:z-power-harmonic}
	Let $z=(z_1,\dots,z_m)$, $z_i=(x_{2i-1},x_{2i})$, then for any multi-index $\alpha=(\alpha_1,\dots,\alpha_m)$ the functions $\Re z^{\alpha}$ and $\Im z^{\alpha}$
	are harmonic in~$x$.
\end{lemma}
\begin{proof}
  Write $z=(z_1,\tilde{z})$ and $\tilde{\alpha} = (\alpha_2,\dots, \alpha_m)$. Then $z^\alpha = z_1^{\alpha_1}(\tilde{z})^{\tilde{\alpha}}$ and
	\begin{align}
		\operatorname{Re}(z^\alpha)=\operatorname{Re}(z_1^{\alpha_1})\operatorname{Re}(\tilde{z}^{\tilde{\alpha}})-\operatorname{Im}(z_1^{\alpha_1})\operatorname{Im}(\tilde{z}^{\tilde{\alpha}}).
	\end{align}
	Thus, because $\tilde{z}$ contains no $x_1$ and no $x_2$ and because $z_1^{\alpha_1}$ is holomorphic we get
	\begin{align}
		(\partial_{x_1}^2+\partial_{x_2}^2)\operatorname{Re}(z^\alpha)=0.
	\end{align}
  Similarly, we get $(\partial_{x_{2j-1}}^2+\partial_{x_{2j}}^2)\operatorname{Re}(z^\alpha)=0$ for $j=1,\dots, m$. Summing up we get $\Delta (\operatorname{Re}(z^\alpha))=0$. The proof for~$\Im(z^\alpha)$ is analogously.
\end{proof}
We can now construct our polynomials~$h$.
\begin{theorem}
  \label{thm:higher-order}
  Let $n \in \setN_{>0}$ be even and let $k \geq 2$. Then there exists a $k$-homogeneous polynomial $h \,:\, \RRn \to \RRN$ with $N=2 \binom{n/2+k-1}{k}$.
\end{theorem}
\begin{proof}
  Let $n=2m$. We change to complex number using the representation $z_j = x_{2j-1} + i x_{2j}$ for $j=1,\dots,m$. Moreover, let $z = (z_1,\dots, z_m) \in \setC^m$ and $\abs{z}^2 = \abs{z_1}^2 + \dots + \abs{z_m}^2$.
  Then
  \begin{align*}
    \abs{x}^{2k} &= \abs{z}^{2k} = \big( \abs{z_1}^2 + \dots + \abs{z_m}^2\big)^k
    = \sum_{\abs{\alpha}=j} \binom{m}{\alpha} \prod_{j=1}^m \abs{z_j}^{2\alpha_j}
    = \sum_{\abs{\alpha}=j} \binom{m}{\alpha} \biggabs{ \prod_{j=1}^m z_j^{\alpha_j}}^2.
    \\
    &=  \sum_{\abs{\alpha}=j} \binom{m}{\alpha} \Bigg(\Re\bigg( \prod_{j=1}^m z_j^{\alpha_j}\bigg) \Bigg)^2 + \sum_{\abs{\alpha}=j} \binom{m}{\alpha} \Bigg(\Im\bigg( \prod_{j=1}^m z_j^{\alpha_j}\bigg) \Bigg)^2,
  \end{align*}
  where $\Re$ and $\Im$ denote the real and imaginary part. As components of~$h$ we take the real and imaginary parts of the $\prod_{j=1}^m z_j^\alpha$. Due to Lemma \ref{lem:z-power-harmonic} these are harmonic. This concludes the proof.
\end{proof}
\begin{remark}
  \label{rem:higherorder-n=2}
  If $n=2$, then $N=2 \binom{n/2+k-1}{k} = 2$. Thus, we found in Theorem~\ref{thm:higher-order} a family of $p$-harmonic functions $u_k\,:\, \RR^2 \to \RR^2$ with regularity and homogeneity given by~\eqref{eq:gamma2}. This is similar to the family of $p$-harmonic functions of higher order constructed in \cite{Aronsson89representation}. Unfortunately, if $n \geq 3$, then $N$ depends also on~$k$.
\end{remark}


\subsection{\texorpdfstring{Examples of $\infty$-harmonic maps}{Examples of infinity-harmonic maps}}\label{sec:infinity_laplacian}

It was observed in~\cite{BhattacharyaDiBenedettoManfredi89} that the homogeneous $p$-Laplacian equation turn for $p \to \infty$ into the $\infty$-Laplacian equation. Indeed, if we start with the $p$-Laplace operator applied to a function $u\,:\, \RRn \to \RRN$, then
\begin{align*}
  \Delta_p u &= \divergence \big( \abs{\nabla u}^{p-2} \nabla u\big) = (p-2) \abs{\nabla u}^{p-4} \Delta_\infty u + \abs{\nabla u}^{p-2} \Delta u,
\end{align*}
where
\begin{align}
  \label{eq:Delta-infty}
  \Delta_\infty u &= \tfrac 12 \sum_{j,k} \partial_j (\abs{\nabla u}^2) \partial_j u.
\end{align}
Now, formally dividing the equation $-\Delta_p u = 0$ by $(p-2)\abs{\nabla u}^{p-4}$ and passing to the limit~$p\to \infty$, we arrive at
\begin{align*}
  -\Delta_\infty u &= 0.
\end{align*}
Another way to obtain~$\Delta_\infty$ at least for $N=1$ is to look at local minimizers of the functional $\norm{\nabla u}_\infty$, see~\cite{Aronsson1967extension}. For the infinity Laplacian the concept of solution is delicate. Most often, viscosity solutions are used, see \cite{Lindqvist16}. If $n=2$ and $N=1$, then $\infty$-harmonic functions are~$C^{1,\alpha}$, see~\cite{EvansSavin08}. They are conjectured to be in $C^{1,\frac 13}$, which is the regularity of the $\infty$-harmonic function $x^{4/3} -y^{4/3}$.
There is also some literature on $\infty$-harmonic maps, i.e. $N \geq 2$. We refer to~\cite{OuTroutmanWilhelm2012,Katzourakis2017} and its references.

It turns out that are examples of Theorem~\ref{thm:cond_for_h} are for $p=\infty$ in fact solutions to~$-\Delta_\infty u=0$.
\begin{theorem}
  \label{thm:Delta-infty}
  Let $n \geq 2$ and let~$h$ and~$k$ be as in Theorem~\ref{thm:cond_for_h}. Then $u(x)= \abs{x}^{1-k} h(x)$ solve $\Delta_\infty = 0$ on~$\RRn$ in the sense of~\eqref{eq:Delta-infty}. Moreover, $\abs{\nabla u} = 1+k(k+n-2)$ and $u \in W^{1,\infty}_{\loc}(\RRn)$.
\end{theorem}
\begin{proof}
  Let $\gamma=\gamma(p,k,n)$ be as in Theorem~\ref{thm:cond_for_h}. Then $\lim_{p \to \infty} \gamma(p,k,n) = 1$. Thus, by~\eqref{eq:abs-nabla-u} we have $\abs{\nabla u}^2 = 1+k(k+n-2)$. In particular, $\abs{\nabla u}^2$ is constant. This immediately implies that $\Delta_\infty u = 0$ and $u \in W^{1,\infty}_{\loc}(\RRn)$. Note that $u$ is 1-homogeneous.
\end{proof}
Let us write down the example for $n=N=2$ and $k=2$ implicitly:
\begin{align*}
  u(x) &= \abs{x}^{-1} 
  \Bigg(\begin{matrix}
    x_1^2-x_2^2
    \\
    2x_1x_2
  \end{matrix}\Bigg)
  .
\end{align*}
This function is just in~$W^{1,\infty}_{\loc}(\RRn)$ but not in~$C^1(\RRn)$. Thus, our example is more irregular then the $\infty$-harmonic map $\phi(x_1,x_2,x_3,x_4) = (x_1^{4/3} - x_2^{4/3},x_3^{4/3} -x_4^{4/3}) \in C^{1,\frac 13}(\RRn)$  given in~\cite[Example 4.5]{OuTroutmanWilhelm2012}.


\printbibliography
\end{document}